\documentclass[twoside,a4paper,12pt]{article}

\usepackage{amsmath}
\usepackage{amsfonts}
\usepackage{amssymb}
\usepackage{tikz}
\usetikzlibrary{matrix}
\usetikzlibrary{arrows,backgrounds,decorations,positioning,fit,petri}
\usepackage{ntheorem}
\usepackage{color}
\usepackage{multirow}
\usepackage{abstract}
\usepackage{pdflscape}
\usepackage[numbers]{natbib}
\usepackage{fancyhdr}
\usepackage[top=1.5in, bottom=1.5in, left=1in, right=1in]{geometry}
\usepackage{graphicx}
\usepackage{caption}
\usepackage{subcaption}
\def\ot{\otimes}
\def\ncp#1#2{#1\langle #2\rangle}
\def\ep{\epsilon}
\def\scal#1#2{\langle #1\bv#2 \rangle}
\def\bv{\mid}

\newcommand{\beqa}{\begin{eqnarray}}
\newcommand{\eeqa}{\end{eqnarray}}

\newtheorem{definition}{Definition}[section]
\newtheorem{proposition}[definition]{Proposition}
\newtheorem{theorem}[definition]{Theorem}
\newtheorem{lemma}[definition]{Lemma}
\newtheorem{remark}[definition]{Remark}

\newtheorem{example}[definition]{Example}

\newenvironment{proof}{{\rm Proof.}\noindent}{\hfill$\square$}

\def\B{\mathcal{B}}
\def\H{\mathcal{H}}
\def\KK{\mathbb{K}}
\def\NN{\mathbb{N}}

\def\Cal #1{{\mathcal #1}}
\long\def\ignore#1{}

\newcommand{\bea}{\begin{eqnarray}}
\newcommand{\eea}{\end{eqnarray}}
\newcommand{\redtext}[1]{\textcolor{red}{[[#1]]}}

\begin{document}
\parindent0pt
\noindent
\vskip 1cm
\title{A combinatorial non-commutative Hopf algebra of graphs}
\author{
{\sf   G\'erard H. E. Duchamp${}^{a}$\thanks{e-mail: ghed@lipn.univ-paris13.fr}, }
{\sf   Lo\"\i c Foissy${}^{b}$\thanks{e-mail: foissy@lmpa.univ-littoral.fr},}
{\sf   Nguyen Hoang-Nghia${}^{a}$\thanks{e-mail: nguyen.hoang@lipn.univ-paris13.fr}, }
\\
{\sf   Dominique Manchon${}^{c}$\thanks{e-mail: manchon@math.univ-bpclermont.fr} }
 and
{\sf Adrian Tanasa${}^{a,d}$\thanks{e-mail: adrian.tanasa@ens-lyon.org}}}
\date{\today}
\maketitle

\vskip-1.5cm

\vspace{2cm}
\thispagestyle{empty}
\begin{abstract}
\begin{center}
\begin{minipage}{.8\textwidth}
A 
non-commutative, planar,  Hopf algebra of rooted trees
was proposed by one of the authors in \cite{Foi02}.
In this paper we propose such a non-commutative Hopf algebra for graphs. 
In order to define a non-commutative product 
we use a quantum field theoretical (QFT) idea, namely the one of introducing 
discrete scales on each edge of the graph (which, within the QFT framework, corresponds to 
energy scales of the associated propagators).
\end{minipage}
\end{center}
\end{abstract}

\vspace{2cm}
{\it AMS classification:} 05E99, 05C99, 16T05, 57T05 

{\it Keywords:} noncommutative Hopf algebras, graphs, discrete scales

\newpage
\section{Introduction}



The Hopf algebra of rooted forests first appeared in the work of A. D\"ur \cite{D86} (and its group of characters, known as the \textsl{Butcher group},  appeared even earlier in the work of J. Butcher in numerical analysis \cite{B72}). It has been rediscovered by D. Kreimer in the context of quantum field theory \cite{K98}, see also \cite{Brou00}. A noncommutative version, using ordered forests of planar trees, has been discovered independently by L. Foissy \cite{Foi02} and R. Holtkamp \cite{H03}. Remarkably enough, this Hopf algebra is self-dual. Commutative Hopf algebras of graphs have been introduced and studied by A. Connes and D. Kreimer \cite{CK98, CK00, CK01}, as a powerful algebraic tool unveiling the combinatorial structure of renormalization.\\

Inspired by constructive quantum field theory \cite{book-rivasseau}, 
we propose in this article a noncommutative version of a Hopf algebra of graphs, by putting a total order on the set of edges. This can be visualized by putting pairwise distinct decorations on each edge, where the decorations take values in the positive integers (or even in any totally ordered infinite set). We prove that the vector space freely generated by these \textsl{totally assigned graphs} (TAGs) is a Hopf algebra. The product is given by the disjoint union of graphs with the ordinal sum order on the edges (see Formula \eqref{eq:product}), and the coproduct is given by Formula \eqref{eq:coprod}, involving subgraphs and contracted graphs.
\\

It is interesting to notice that what we call here TAGs have already been analyzed, 
from a completely different perspective (the travelling salesman problem), 
by O. Boruvka, already in 1926 (see \cite{Bor26}). 
He proved that the shortest spanning tree of such a graph is unique\footnote{Note however that within 
this travelling salesman context the decoration associated to a self-loop (known as a {\it tadpole edge} in quantum 
field theoretical language) is zero, while in our context one can have strictly positive integers associated to such
self-loops.}.
Moreover, the same problem was solved through several simpler explicit constructions
by the celebrated Kruskal algorithm \cite{Kru55}. 
\\

Let us mention here that, throughout this paper, we do not deal with graphs which 
are necessarily 1-particle-irreducible (\textsl{i. e.} bridge-less). Moreover we do not consider
in this paper \textsl{external edges}, as it is done in quantum field theory.

\section{Why discrete scales?}



As already announced above, the idea of decorating the edges of a graph with discrete scales 
comes from quantum field theory, or more precisely 
from the multi-scale analysis technique used in perturbative and in constructive 
renormalization
 (see Vincent Rivasseau's book \cite{book-rivasseau}).

In quantum field theory each edge of a graph is associated to a {\it propagator}
$C=1/H$ (which, in elementary particle physics represents a particle).
Introducing discrete scales comes to a ``slicing'' of the propagator
\bea  C &=& \int_0^\infty  e^{- \alpha H}  d \alpha \; ,
\;   \sum_{i=0}^{\infty}  C^{i}   \label{decoab1}  \\
C^{a} &=&  \int_{M^{-2a}}^{M^{-2(a-1)}} e^{- \alpha H}  d \alpha \; ,
\; C^{0} =  \int_{1}^{\infty} e^{- \alpha H}  d \alpha .
\eea
When some discrete integer $a$ is associated to a given edge, this 
means that the propagator assigned to this edge lies within a given energy scale.
One thus introduces more information 
(replacing graphs by ``assigned graphs'') which yields in turn some refinement of the analysis, as we will explain here.\\

When integrating over the energy scales of the internal propagators in a Feynman graph in quantum field 
theory, one obtains
the {\it Feynman integral} associated to the respective graphs.
Usually, these integrals are divergent. This is when {\it renormalization} comes in,
subtracting (when possible) the divergent parts of these Feynman integrals, in 
a self-consistent way (see again Vincent Rivasseau's book \cite{book-rivasseau} or any other 
textbook on renormalization).
Nevertheless, these divergences only appear for high energies 
(the so-called {\it ultraviolet regime})\footnote{Divergences for low energies
(the {\it infrared regime}) can also appear in quantum field theory, but 
one can deal with this type of divergences in a different way. This lies outside the 
purpose of this section.}, 
which corresponds, within the multi-scale formalism, to 
the case when all the integer scales associated to the internal edges are higher then the edges 
associated to the external edges (see again Vincent Rivasseau's book \cite{book-rivasseau} for details).

When dealing with this divergence subtraction (the subtraction of the so-called 
``counterterms"), an important ``technical'' 
complication is given by the issue of ``overlapping divergences'', which 
is given by overlapping subgraphs which lead, independently, to divergences.
This problem is solved in an elegant way within the multi-scale analysis, where all
subgraphs leading to divergences are either disjoint or nested.

Let us also emphasize that the multi-scale renormalization technique 
splits the counterterms into two categories: ``useful" and ''useless"
counterterms (the useful ones being the ones corresponding to subgraphs where 
all the integer scales associated to the internal edges are higher then the edges 
associated to the external edges). This refining is not possible 
without the scale decoration of edges; furthermore, it also solves another 
major problem
of renormalization, the so-called ``renormalon problem"
(which is an issue when one wants to sum over the contribution of 
each term in perturbation theory).

\medskip

This versatile technique of multi-scale analysis 
was successfully applied for scalar quantum field theory renormalization
(see again \cite{book-rivasseau}), 
the condensed matter case  \cite{BG1},\cite{FT1},\cite{Rivasseau:2011ri},  renormalization of 
scalar quantum field theory on the non-commutative Moyal space
(see  \cite{GMRT}, \cite{4men}, \cite{propa}, \cite{GW} and \cite{gn})
and recently to the renormalization of quantum gravity tensor models \cite{BGR},\cite{COR}.

\medskip

The combinatorics of the multi-scale renormalization was encoded in a 
Hopf algebraic framework in \cite{KRT12}. As already announced above,
the Hopf algebraic setting of \cite{KRT12} is commutative, 
and the assigned graphs designed there can have equal scale integers for several edges
of the same graph.

\section{Non-commutative graph algebra structure}\label{sec:alg}


In this section 
we define the space of totally assigned graphs (TAG) and a non-commutative 
algebra structure on this space. 


\begin{definition}
A {\bf totally ordered scale assignment} $\mu$ for a graph $\Gamma$ is 
a total order on the set $E(\Gamma)$ of edges 
of $\Gamma$.
\end{definition}
It will be convenient to visualize the total order $\mu$ by choosing a compatible labelling, i.e. an injective increasing map from $\big(E(\Gamma),\mu\big)$ into $\mathbb N^*=\{1,2,3,\ldots\}$. There is of course an infinite number of possible labellings. 
The unique such map with values in $\{1,\ldots,|E(\Gamma)]\}$ will be called 
the \textsl{standard labelling} associated with $\mu$.

\begin{example} An example of a totally ordered scale assignment with nonstandard labelling is given in Fig. 
\ref{fig:TAG}.

\begin{figure}[!ht]
\begin{center}
\includegraphics[scale=1.5]{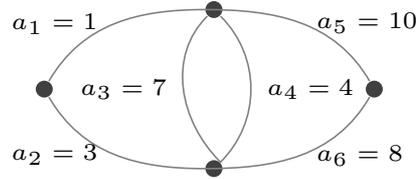}
\caption{A graph with a totally ordered scale assignment.}
\label{fig:TAG}
\end{center}
\end{figure}
\end{example}

\begin{definition}
A {\bf totally assigned graph} (TAG) is a pair  $(\Gamma,\mu)$ formed 
by a graph $\Gamma$ (not necessarily connected), together with a totally ordered scale assignment $\mu$.
\end{definition}


\medskip

Consider now a field $\KK$ of characteristic $0$, and let ${\Cal H}$ be 
the $\KK$- vector space freely spanned by TAGs. 
The {\bf product} $m$ on $\H$ is given by:
\begin{equation}\label{eq:product}
m\big((\Gamma_1,\mu),(\Gamma_2,\nu)\big)=(\Gamma_1,\mu)\cdot(\Gamma_2,\nu):=(\Gamma_1\sqcup\Gamma_2,\mu\sqcup\nu),
\end{equation}
where $\Gamma_1\sqcup\Gamma_2$ is the disjoint union of the two graphs,
and where $\mu\sqcup\nu$ is the \textsl{ordinal sum order}, i.e. the unique total order on $E(\Gamma_1)\sqcup E(\Gamma_2)$ 
which coincides with $\mu$ (resp. $\nu$) on $\Gamma_1$ (resp. $\Gamma_2$), and 
such that $e_1<e_2$ for any $e_1\in\Gamma_1$ and $e_2\in\Gamma_2$. 
Although the disjoint union of graphs is commutative, 
the product is not because the total orders $\mu\sqcup\nu$ and $\nu\sqcup\mu$ are different
(see also Remark \ref{remarca} below).
Associativity is however obvious. The empty TAG is the empty graph, denoted by $1_{\H}$.

\begin{figure}[!ht]
\centering
        \begin{subfigure}[b]{0.3\textwidth}
                \centering
\includegraphics[scale=1.3]{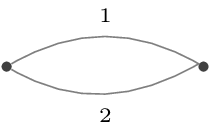}
\caption{The TAG $(\Gamma_1,\mu_1)$.}\label{fig:TAG1}
 \end{subfigure}
 \hspace{1cm}
 \begin{subfigure}[b]{0.3\textwidth}
 \centering
 \includegraphics[scale=1.1]{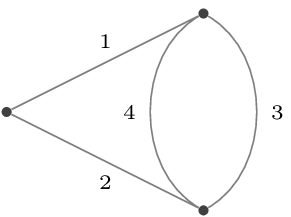}
\caption{The TAG $(\Gamma_2,\mu_2)$.}\label{fig:TAG2}
 \end{subfigure}
\caption{Two examples of TAGs.} \label{fig:expleTAG}
\end{figure}

\begin{example}
Let $(\Gamma_1,\mu_1)$ and $(\Gamma_2,\mu_2)$ be the two graphs in Fig. \ref{fig:expleTAG}.

One has
\begin{equation*}
\label{exemplu}
m\big ((\Gamma_1,\mu_1),(\Gamma_2,\mu_2)\big) = \parbox{7cm}{\includegraphics[scale=1]{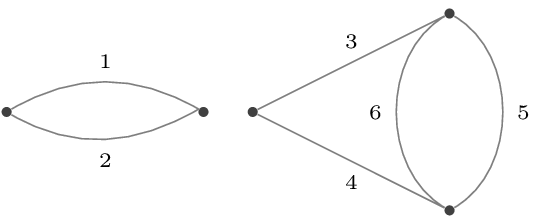}
}.
\end{equation*}
\end{example}
Summing up:
\begin{proposition}
$(\H,m,1_{\H})$ is an associative unitary algebra.
\end{proposition}

\begin{remark}
\label{remarca}
The standard labelling of the product $(\Gamma_1,\mu_1)\cdot(\Gamma_2,\mu_2)$ is obtained by keeping the standard labelling for $E(\Gamma_1)$ and shifting  the standard labelling of $E(\Gamma_2)$ by $|E(\Gamma_1)|$.
\end{remark}

Let us end this section by the following example illustrating the non-commutativity of our product:

\begin{example}
 One has 
\begin{equation}
 m \left( \parbox{3.3cm}{\includegraphics[scale=.7]{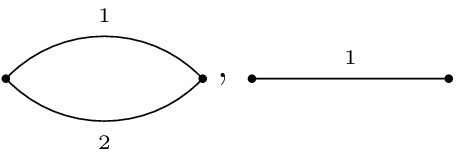}} \right) = 
\parbox{3cm}{\includegraphics[scale=.7]{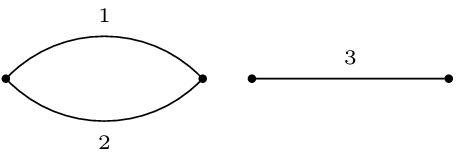}}
\end{equation}
 and
\begin{equation}
 m\left( \parbox{3.3cm}{\includegraphics[scale=.7]{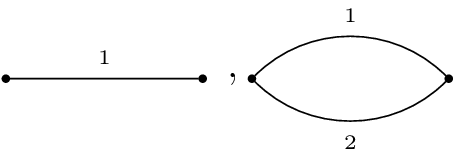}} \right) = \parbox{3.3cm}{\includegraphics[scale=.7]{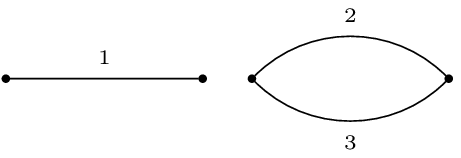}} = \parbox{3.3cm}{\includegraphics[scale=.7]{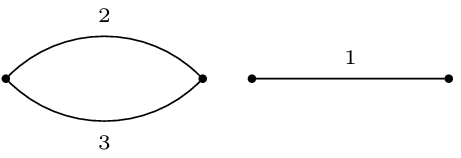}}.
\end{equation}

\end{example}



\section{Hopf algebra structure}\label{sec:hop}



Let us first give the following definitions:

\begin{definition}
\label{def:subgraf}
 A subgraph $\gamma$ of a graph $\Gamma$ is the graph formed by a given subset of edges $e$ 
of the set of edges of the graph $\Gamma$ together with the vertices that the edges of $e$ 
hook to in $\Gamma$.
\end{definition}

 Let us notice that a subgraph is not necessary connected
nor spanning. 

\begin{definition}
A {\bf totally assigned subgraph} $(\gamma,\nu)$ of a given TAG  $(\Gamma,\mu)$ is a subgraph $\gamma$ of $\Gamma$in the sense of Definition \ref{def:subgraf}, together with the total order $\nu$ on $E(\gamma)$ induced by $\mu$. The {\bf shrinking} $(\Gamma,\mu)/(\gamma,\nu)$ of a given TAG $(\Gamma,\mu)$ by a totally assigned subgraph $(\gamma,\nu)$ is defined as follows: the cograph $\Gamma/\gamma$ is obtained as usual, by shrinking each connected component of $\gamma$ on a point, and the totally ordered scale assignment $\mu/\nu$ of the cograph $\Gamma/\gamma$ 
is given by restricting the total order $\mu$ on the edges of the cograph, i.e. the edges of $\Gamma$ which are not internal to $\gamma$. 
The TAG $(\Gamma/\gamma,\mu/\nu)$ is called a totally assigned cograph.
\end{definition}




\medskip
Let us now define the coproduct
$\Delta: \H \longrightarrow \H \otimes \H$ as 
\begin{equation}\label{eq:coprod}
\Delta\big((\Gamma,\mu)\big) =  \sum_{\emptyset \subseteq(\gamma,\nu)\subseteq 
(\Gamma,\mu)}  (\gamma,\nu) \otimes  (\Gamma/\gamma,\mu/\nu)
\end{equation}
for any TAG $(\Gamma,\mu)$.
\begin{example}
1) Let $(\Gamma_1,\mu_1)$ be the TAG in Fig. \ref{fig:TAG1}.

One has the coproduct:

\begin{eqnarray*}
\Delta(G,\mu) = (G,\mu)\otimes 1_\H + 1_\H \otimes (G,\mu) + 
2\parbox{3.5cm}{\includegraphics[scale=1]{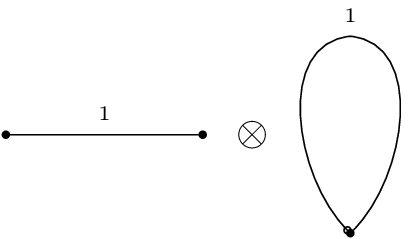}}
\end{eqnarray*}

2) Let $(G,\mu)$ be the TAG given in Figure \ref{fig:TAG2}. 

\begin{align*}
\Delta&((G,\mu)) = (G,\mu) \otimes 1_{\H} + 1_{\H} \otimes (G,\mu) + 2 \parbox{4cm}{\includegraphics[scale=1]{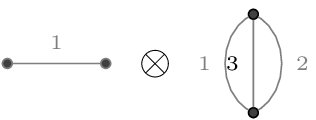}}  \nonumber\\&  
+ 2 \parbox{3.5cm}{\includegraphics[scale=1]{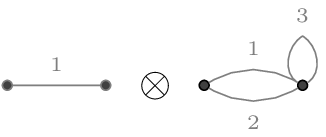}}  
+ 5 \parbox{2.5cm}{\includegraphics[scale=.9]{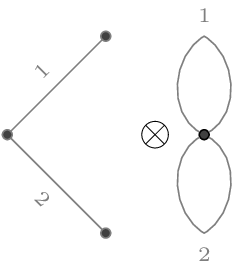}} 
\nonumber\\&
+ \parbox{3.5cm}{\includegraphics[scale=1]{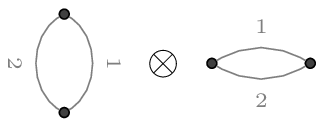}} + 2 \parbox{3cm}{\includegraphics[scale=.8]{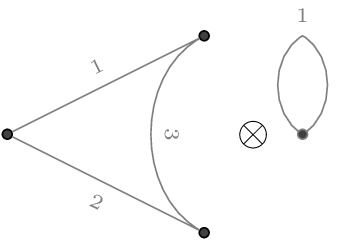}} 
\nonumber\\&  
+ 2 \parbox{3cm}{\includegraphics[scale=1]{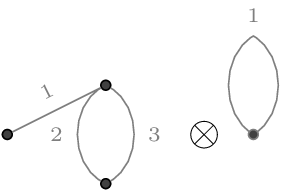}} 
\end{align*}
\end{example}

\begin{example}
Let $(G,\mu)$ be the TAG given in Fig. \ref{fig:TAG4}. 

\begin{figure}[!ht]
\begin{center}
\includegraphics[scale=1]{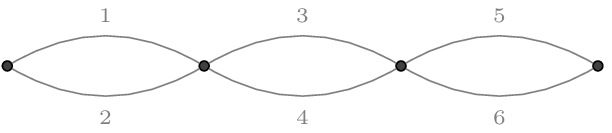}
\caption{Another TAG.}
\label{fig:TAG4}
\end{center}
\end{figure}

\begin{eqnarray*}
\Delta((G,\mu)) = (G,\mu)\otimes 1_{\H} + 1_{\H} \otimes (G,\mu) + \parbox{6cm}{\includegraphics[scale=.8]{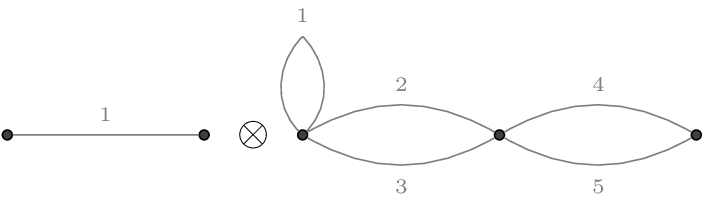}} \\
+ \parbox{6cm}{\includegraphics[scale=.8]{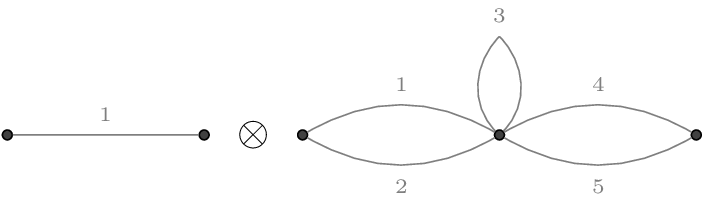}} + \dots + 3 \parbox{6cm}{\includegraphics[scale=.8]{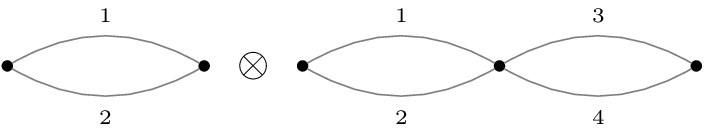}}  \\ 
+ \parbox{5cm}{\includegraphics[scale=.7]{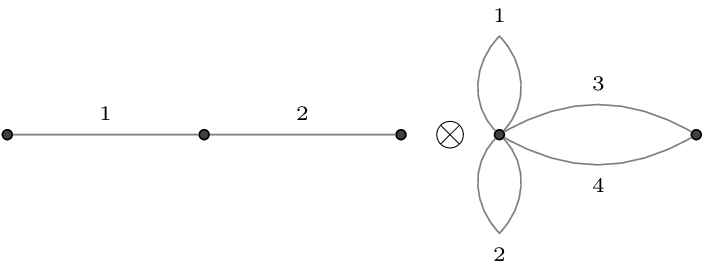}} + \dots + \parbox{6.5cm}{\includegraphics[scale=.8]{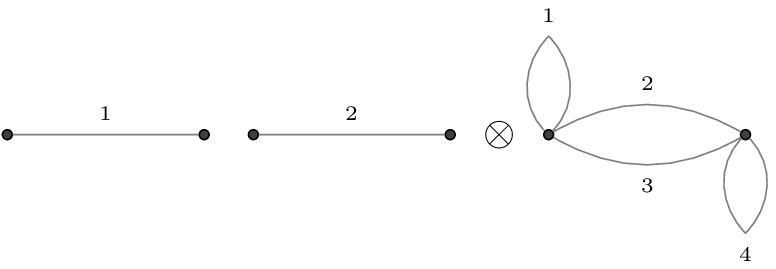}} + \dots \\ \parbox{5cm}{\includegraphics[scale=.7]{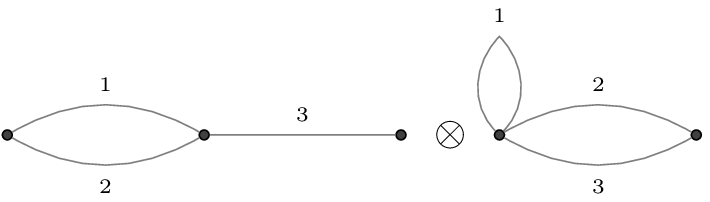}} + \dots + \parbox{6cm}{\includegraphics[scale=.7]{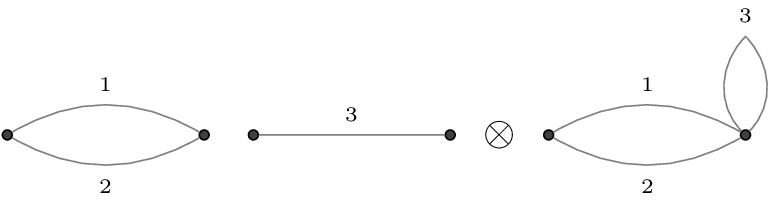}} + \dots \\
\parbox{6cm}{\includegraphics[scale=.7]{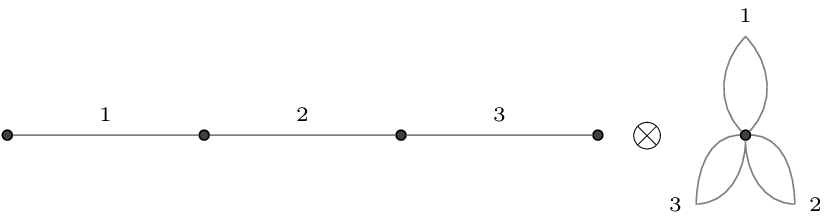}} + \dots + 2 \parbox{5.5cm}{\includegraphics[scale=.7]{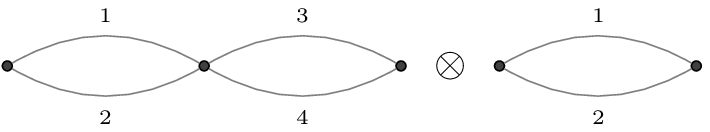}} 
\\ + 2 \parbox{5cm}{\includegraphics[scale=.7]{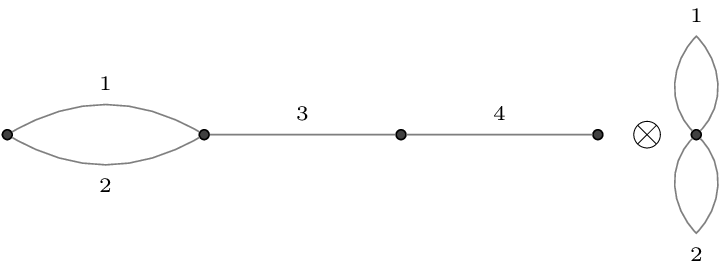}}  + \parbox{5cm}{\includegraphics[scale=.7]{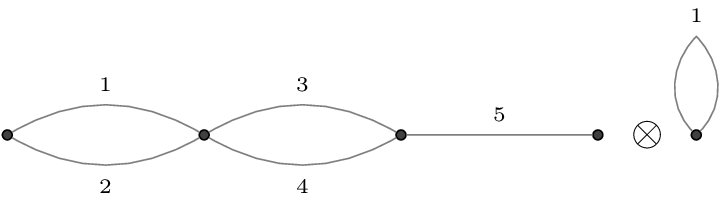}} + \dots 
\end{eqnarray*}
where we have denoted by $\dots$ for the similar graph obtained from choosing the  same type of subgraph on the LHS on the coproduct; note that this can lead, in some cases, 
to distinct subgraphs on the RHS (sometimes just because of the distinct labels).
\end{example}


\medskip
\begin{lemma}\label{lm:coass}
Let $(\Gamma,\mu)$ be a TAG in $\H$. Let $(\gamma,\nu)$ and $(\delta,\nu')$ 
be two totally assigned subgraphs such that 
$(\delta,\nu') \subseteq (\gamma,\nu) \subseteq (\Gamma,\mu)$. 
Then, one has 
\begin{equation}
 (\Gamma/\gamma,\mu/\nu) =  \big((\Gamma/\delta)/({\gamma / \delta}),(\mu/\nu')/(\nu / \nu')\big).
\end{equation}
\end{lemma}
\begin{proof}
Since, $\delta \subseteq \gamma \subseteq \Gamma$, then one has 
$\gamma / \delta \subseteq \Gamma / \delta$. 
One has $\Gamma / \gamma = (\Gamma /  \delta) / (\gamma / \delta)$. 
Moreover, the total order $\nu$ (resp. $\nu'$) is induced by restriction of $\mu$ 
(resp. $\nu$ or $\mu$) to the set of edges of $\delta$.  
Then $\mu / \nu = (\mu/ \nu') / (\nu / \nu')$, which concludes the proof.

\end{proof}

\begin{proposition}\label{prop:coass}
The coproduct defined in \eqref{eq:coprod} is coassociative.
\end{proposition}

\begin{proof}
Let $(\Gamma,\mu) \in \H$. Then, one has:
\begin{align}\label{eq:coassLHS}
&(\Delta \otimes Id) \circ \Delta\big ((\Gamma,\mu)\big) = 
(\Delta \otimes Id) 
\left(\sum_{(\gamma,\nu)\subset (\Gamma,\mu)}  (\gamma,\nu) \otimes 
 (\Gamma/\gamma,\mu/\nu)\right) \nonumber\\
&= \sum_{(\gamma,\nu)\subset (\Gamma,\mu)} \left(\sum_{(\gamma',\nu')\subset (\gamma,\nu)} 
 (\gamma',\nu') \otimes  (\gamma/{\gamma'},\nu/{\nu'})\right) 
\otimes  (\Gamma/\gamma,\mu/\nu)  \nonumber\\
&= \sum_{\substack{(\gamma,\nu)\subset (\Gamma,\mu) \\ 
(\gamma',\nu')\subset (\gamma,\nu)}}  
 (\gamma',\nu') \otimes  (\gamma/{\gamma'},\nu/{\nu'}) 
\otimes  (\Gamma/\gamma,\mu/\nu).
\end{align}
and
\begin{align}\label{eq:coassRHS}
&(Id \otimes \Delta) \circ \Delta ((\Gamma,\mu)) = (Id \otimes \Delta) 
\left(\sum_{(\gamma,\nu))\subset (\Gamma,\mu)} 
 (\gamma,\nu) \otimes  (\Gamma/\gamma,\mu/\nu)\right) \nonumber\\
&= \sum_{(\gamma,\nu)\subset (\Gamma,\mu)}  (\gamma,\nu) 
\otimes \left( \sum_{(\gamma',\nu')\subset (\Gamma/\gamma,\mu/\nu)}  (\gamma',\nu') 
\otimes  \big((\Gamma/\gamma)/{\gamma'},(\mu/\nu)/{\nu'}\big)\right).
\end{align}

There is a one-to-one correspondence between the 
assigned subgraphs $(\gamma',\nu')\subseteq (\Gamma/\gamma,\mu/\nu)$ 
and the 
assigned subgraphs $(\gamma_1,\nu_1)\subseteq (\Gamma,\mu)$ 
such that $(\gamma,\nu)\subseteq(\gamma_1,\nu_1)$. Indeed, starting from an 
assigned subgraph $(\gamma_1,\nu_1)\subseteq (\Gamma,\mu)$ 
such that $ (\gamma,\nu)\subseteq (\gamma_1,\nu_1)$, we find 
an   
assigned subgraph $(\gamma',\nu')\subseteq (\Gamma/\gamma,\mu/\nu)$ by restricting the total order $\nu_1$ to the edges of $\gamma_1$ which are not internal to $\gamma$, and the inverse operation consists in extending the total order $\nu'$ to all edges of $\gamma_1$ in the unique way compatible with the total order $\mu$ on $E(\Gamma)$. 

Applying Lemma \ref{lm:coass}, one has $\big((\Gamma/\gamma)/{\gamma'},(\mu/\nu)/{\nu'}\big) = (\Gamma/\gamma_1,\mu/\nu_1 )$. Equation \eqref{eq:coassRHS} can then be rewritten as follows:
\begin{align}\label{eq:coassRHS1}
&(Id \otimes \Delta) \circ \Delta \big((\Gamma,\mu)\big) \nonumber\\
& = \sum_{\substack{(\gamma_1,\nu_1)\subset (\Gamma,\mu) \\ (\gamma,\nu)\subset (\gamma_1,\nu_1)}}  
 (\gamma,\nu) \otimes  (\gamma_1/{\gamma},\nu_1/{\nu}) 
\otimes  (\Gamma/{\gamma_1},\mu/{\nu_1}).
\end{align}
Using equations \eqref{eq:coassLHS} and \eqref{eq:coassRHS1}, one concludes the proof.
\end{proof}

\medskip
Furthermore, we define the counit $\epsilon: \H \longrightarrow \KK$ by 
\begin{equation}
\epsilon\big((\Gamma,\mu)\big) = \begin{cases} 1 \mbox{ if } (\Gamma,\mu) = 1_{\H};\\0 \mbox{ otherwise.}\end{cases}
\end{equation}
\medskip
\begin{theorem}
The triple $(\H,\Delta,\epsilon)$ is a coassociative coalgebra with counit.
\end{theorem}

\begin{proof}
Let us show that $\epsilon$ is a counit of the coalgebra. For any 
TAG $(\Gamma,\mu)$, one has
\begin{eqnarray*}
(\epsilon \otimes Id) \circ \Delta\big ((\Gamma,\mu)\big) = \epsilon\big((\Gamma,\mu)\big)1 + 
\epsilon(1_{\H})(\Gamma,\mu) + \sum_{(\gamma,\nu) \subsetneq (\Gamma,\mu)} \epsilon(\gamma,\nu) 
 (\Gamma/\gamma,\mu/\nu) = (\Gamma,\mu).
\end{eqnarray*}

Analogously, one has: $(Id \otimes \epsilon) \circ \Delta\big ((\Gamma,\mu)\big) = (\Gamma,\mu)$. 
One thus concludes that the maps $Id$, $(\epsilon \otimes Id) \circ \Delta$ and 
$(Id \otimes \epsilon) \circ \Delta$ 
coincide on 
TAGs, thus proving that 
$\epsilon$ is a counit of $\Delta$.
Using now Proposition \ref{prop:coass}, one concludes the proof.
\end{proof}
\begin{example}
Let us check the coassociativity of our coproduct on the example of the 
standard labeled TAG of Fig. \ref{fig:BigGraph}.
\begin{figure}
\begin{center}
\includegraphics[scale=1.25]{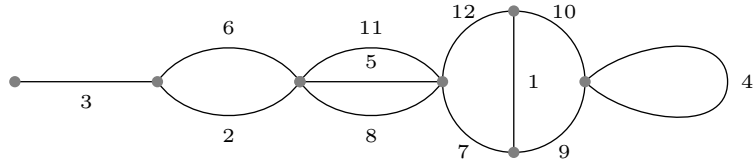}
\caption{A standard labeled TAG.}
\label{fig:BigGraph} 
\end{center}
\end{figure}
When acting with the coproduct on this standard labeled TAG, one gets the term of Fig. 
\ref{fig2}, which obviously adds up to the rest of the coproduct terms.
\begin{figure}
\begin{center}
\includegraphics[scale=1.25]{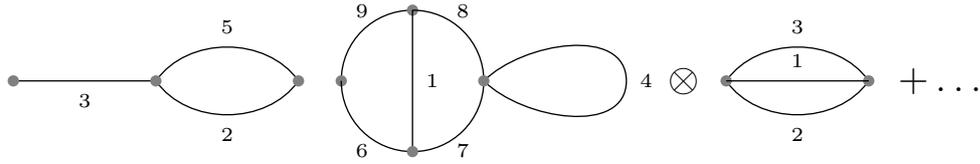}
\caption{One of the terms obtained by acting with the coproduct on the standard labeled TAG of 
Fig \ref{fig:BigGraph}.}
\label{fig2} 
\end{center}
\end{figure}
Another type of term is the one of Fig. \ref{fig3} (which again adds up to the rest of the coproduct terms).
\begin{figure}
\begin{center}
\includegraphics[scale=1.25]{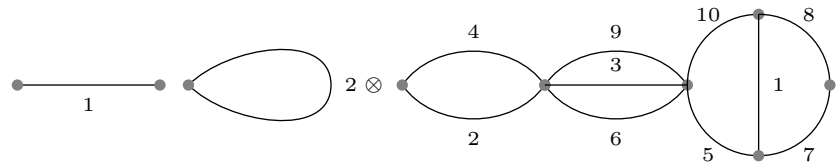}
\caption{Another terms obtained by acting with the coproduct on the standard labeled TAG of 
Fig \ref{fig:BigGraph}.}
\label{fig3} 
\end{center}
\end{figure}
Acting now on these terms with  $(\Delta \otimes Id)$ and respectively with $(Id \otimes \Delta)$
leads to the  same term  represented in Fig. \ref{fig4} with standard labelling.
\begin{figure}
\begin{center}
\includegraphics[scale=1.25]{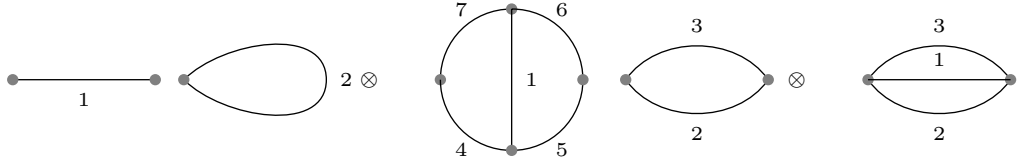}
\caption{The resulting in the LHS and RHS of the coassociativity identity.}
\label{fig4} 
\end{center}
\end{figure}
Let us also emphasize that this term cannot be obtained from other terms of $\Delta$ because of the diagrammatic 
difference of the various disconnected components of the graphs.
\end{example}

\bigskip

One has:

\begin{proposition}\label{prop:bialg}
Let $(\Gamma_1,\mu_1)$ and $(\Gamma_2,\mu_2)$ be two  
TAGs in $\H$. One has
\begin{equation}
\label{eq:coprod1}
\Delta\Big(m\big((\Gamma_1,\mu_1),(\Gamma_2,\mu_2)\big)\Big)
 = m^{\otimes 2}\circ \tau_{23}\big(\Delta(\Gamma_1,\mu_1), \Delta(\Gamma_2,\mu_2)\big)
\end{equation}
where $\tau_{23}$ is the flip of the two middle factors in $\Cal H^{\otimes 4}$.
\end{proposition}

\begin{proof}
One has 
\begin{align}
\Delta\Big(m\big((\Gamma_1,\mu_1),(\Gamma_2,\mu_2)\big)\Big) 
&= \Delta (\Gamma_1 \sqcup \Gamma_2, \mu_1 \sqcup \mu_2)\\
&= \sum_{\emptyset\subseteq(g,\nu)\subseteq (\Gamma_1 \sqcup\Gamma _2,\mu_1 \sqcup \mu_2)} 
 (\gamma,\nu) \otimes  \big((\Gamma_1 \sqcup \Gamma_2)/\gamma,(\mu_1 \sqcup \mu_2)/\nu\big) \nonumber\\
& =  \sum_{\substack{(\gamma_1,\nu_1)\subseteq (\Gamma_1,\mu_1) \\  (\gamma_2,\nu_2)\subseteq 
(\Gamma_2,\mu_2)}} 
 (\gamma_1 \sqcup \gamma_2,\nu_1\sqcup \nu_2) \otimes  (\Gamma_1/\gamma_1 \sqcup \Gamma_2/\gamma_2,\mu_1/\nu_1 \sqcup \mu_2/\nu_2) \nonumber\\
& = m^{\otimes 2} \circ \tau_{23}(\Delta(\Gamma_1,\mu_1),\Delta(\Gamma_2,\mu_2)).
\end{align}
\end{proof}

\begin{example}
Let $(\Gamma_1,\mu_1)$ and $(\Gamma_2,\mu_2)$ be the graph in Fig. \ref {fig:expleTAG}(a).




One has:
\begin{eqnarray*}
\Delta(\Gamma_1 \sqcup \Gamma_2,\mu_1\sqcup\mu_2) = &&(\Gamma_1 \sqcup \Gamma_2,\mu_1\sqcup\mu_2) \otimes 1_\H + 
1_\H \otimes (\Gamma_1 \sqcup \Gamma_2,\mu_1\sqcup\mu_2) \\
&&+ 4 \parbox{4.5cm}{\includegraphics[scale=.6]{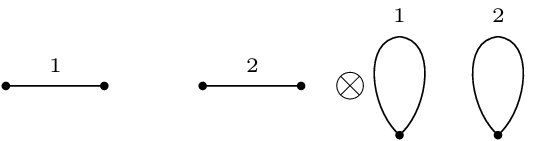}}
 + \parbox{5cm}{\includegraphics[scale=.9]{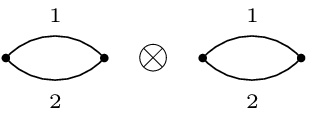}}  \\ 
&& + 4 \parbox{5.2cm}{\includegraphics[scale=1]{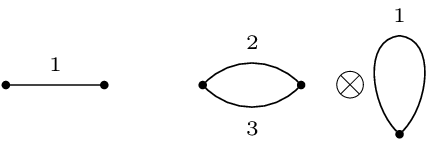}}
\end{eqnarray*}
\end{example}
\medskip

One has:
\begin{theorem}\label{thm:bialg}
$(\H,m,1_{\H},\Delta,\epsilon)$ is a bialgebra.
\end{theorem}

\begin{proof}
Using Proposition \ref{prop:bialg}, it follows that $\Delta$ is a morphism of algebras. 
One thus concludes the proof.
\end{proof}

\bigskip
For all $n \in \NN$, one calls $\H(n)$ the vector space generated by the TAGs with $n$ edges. Then one has $\H = \bigoplus_{n\in\NN}\H(n)$. Moreover, one has:

\begin{itemize}
\item[1.] For all $m,n \in \NN$, $\H(m)\H(n) \subseteq \H(m+n) $.
\item[2.] For all $n\in \NN$, $\Delta\big(\H(n)\big) \subseteq \sum_{i+j=n} \H(i)\otimes \H(j)$.
\end{itemize}

One thus concludes that $\H$ is a {\it graded bialgebra}. Note that 
 $\H$ is {\it connected}.

We can now state the main result of this section:
\begin{theorem}
The bialgebra $(\H,m,1_\H,\Delta,\epsilon)$ is a Hopf algebra.
\end{theorem}
\begin{proof}
The bialgebra $(\H,m,1_\H,\Delta,\epsilon)$ is
 connected and graded. The conclusion follows. The antipode $S :\H \longrightarrow \H$ verifies $S(1_{\H})=1_{\H}$, and is given on 
a non-empty 
TAG $(\Gamma,\mu)$ by one of the two following recursive formulas:
\begin{eqnarray}
S(\Gamma,\mu) &=& - (\Gamma,\mu) - \sum_{\emptyset\subsetneq(\gamma,\nu) \subsetneq (\Gamma,\mu)}  
S(\gamma,\nu)\cdot   (\Gamma/{\gamma},\mu/{\nu})\\
&=& - (\Gamma,\mu) - \sum_{\emptyset\subsetneq(\gamma,\nu) \subsetneq (\Gamma,\mu)}  
(\gamma,\nu)\cdot S(\Gamma/{\gamma},\mu/{\nu}).
\end{eqnarray}
\end{proof}

\bigskip


Note that if one considers now the Hopf algebra  $\H^c$ of graphs 
(without any edge scale decoration), one has:

\begin{proposition}
The map $\pi$ from $\H$ to $\H^c$ defined on the 
TAGs by $\pi\big((\Gamma,\mu)\big) = \Gamma$ 
is a Hopf algebra morphism.
\end{proposition}

\begin{proof}
This statement directly follows from the definitions.
\end{proof}

\bigskip
A further non-commutative Hopf algebra of TAGs can be defined
when considering only graphs of a given quantum field theoretical model 
and defining the coproduct as the appropriate sum on 
the class of superficially divergent graphs (see for example \cite{KRT12},
where such a Hopf algebra was defined, in a commutative setting).\\

Let us end this paper by the following concluding remarks. The non-commutative graph 
Hopf algebraic structure defined here is a combinatorial Hopf algebra (CHA) using 
 a selection/contraction coproduct rule - this type of CHAs being called CHAs of type I in 
\cite{HDR}. 
Examples of such CHAs are the Connes-Kreimer Hopf algebras of QFT \cite{CK98, CK00, CK01}, 
of non-commutative Moyal QFT \cite{moyal2, moyal1}, of quantum gravity spin-foam models 
\cite{sf1, sf2}, of random tensor models \cite{tens}, or the word Hopf algebra WMat \cite{WMat}.
This type of coproduct rule is fundamentally different of the 
selection/deletion rule used in CHAs such as FQSym, MQSym, the Loday-Ronco Hopf algebra... - CHAs of type II 
(see again \cite{HDR}).
It seems however interesting to us to investigate in what circumstances one can find some 
non-trivial mathematical 
relations between these two types of CHAs.

\bigskip

\section*{Acknowledgements}
G. Duchamp, N. Hoang-Nghia and A. Tanasa  acknowledge the "Combinatoire  alg\'ebrique" 
Univ. Paris 13, Sorbonne Paris Cit\'e BQR  grant.
A. Tanasa further acknowledges
the grants PN 09 37 01 02 and CNCSIS Tinere Echipe 77/04.08.2010.

\addcontentsline{toc}{section}{References}
\bibliographystyle{plain}

\vspace{2cm}

\noindent
{\small ${}^{a}${\it Universit\'e Paris 13, Sorbonne Paris Cit\'e, 99, avenue Jean-Baptiste Cl\'ement \\
LIPN, Institut Galil\'ee, 
CNRS UMR 7030, F-93430, Villetaneuse, France, EU}}\\
{\small ${}^{b}${\it Universit\'e du Littoral, LMPA Joseph Liouville, 50, rue Ferdinand Buisson
CS 80699 - 62228 Calais Cedex, France, EU }}\\
{\small ${}^{c}${\it Universit\'e Blaise Pascal - Laboratoire de Math\'ematiques,
CNRS UMR 6620, Campus des C\'ezeaux - BP 80026, F63171 Aubi\`ere Cedex, France, EU }}\\
{\small ${}^{d}${\it 
Horia Hulubei National Institute for Physics and Nuclear Engineering,\\
P.O.B. MG-6, 077125 Magurele, Romania, EU}}\\
\end{document}